\documentclass[reqno,12pt]{amsart}
\usepackage{amsmath, amssymb, amsthm, mathrsfs,color}
\usepackage{graphicx,url}
\usepackage[all]{xy}
\newtheorem{theorem}{Theorem}
\newtheorem{prop}[theorem]{Proposition}
\newtheorem{cor}[theorem]{Corollary}
\newtheorem{lemma}[theorem]{Lemma}

\newcommand{\ZZ}{\mathbb{Z}}

\newcommand{\CC}{\mathbb{C}}
\newcommand{\G}{\Gamma}

\newcommand{\PP}{\mathbb{P}}
\newcommand{\SLZ}{\operatorname{SL}_2(\ZZ)}

\newcommand{\ord}{\operatorname {ord}}
\newcommand{\g}{{\gamma}}
\let\a\alpha      
  \let\g\gamma

\def\Gal{{\rm Gal}}

\def\F{{\mathbb F}}
  
\def\Z{{\mathbb Z}}

\def\Q{{\mathbb Q}}

\def\G{\Gamma}
\def\ord{{\rm ord}}

\begin{document}

\title[Noncongruence Modular Forms II]{On  modular forms for some noncongruence
subgroups of $\SLZ$ II}
\author{Chris Kurth}
\address{Department of Mathematics\\Iowa State University\\Ames, IA 50011 \\USA}
\email{kurthc@iastate.edu}
\author{Ling Long}

\address{Department of Mathematics\\Iowa State University\\Ames, IA 50011 \\USA}
\email{linglong@iastate.edu}
\begin{abstract}

In  this paper we show  two classes of noncongruence subgroups
 satisfy the so-called
unbounded denominator property.  In particular, we establish our
conjecture in \cite{Kurth-Long06} which says that every type II
noncongruence character group of $\G^0(11)$ satisfies the unbounded
denominator property.
\end{abstract}
 \thanks{Both authors are supported in part
by the NSA grant \#H98230-08-1-0076.}

\maketitle

\section{Introduction}
It is well-known that the modular group $\SLZ$ fails to satisfy the
so-called congruence property. As a matter of fact, the majority of
finite index subgroups of the modular group are noncongruence.
Identifying congruence  subgroups of the modular group is a
fundamental question. Although there are explicit algorithms
available for this purpose  \cite{Lang-cong-algorithm, hsu96}, they
require very specific data of the group and hence are not always
effective. Another plausible approach is via the modular forms for
these groups. For instance, if a finite index subgroup $\G$ of the
modular group has genus 0, then knowing that any of its Hauptmoduls
is congruence (in the sense that it is invariant under a congruence
subgroup) is sufficient to conclude that $\G$ is congruence. For
many interesting cases, the coefficients of these Hauptmoduls are
algebraic or combinatorial. A classical example is that the Fourier
coefficients of the modular $j$-function are related to the
dimensions  of the irreducible representations of the monster group.
Consequently, these Hauptmoduls have algebraically integral Fourier
coefficients. A general belief is that a meromorphic modular form
with algebraically integral Fourier coefficients must be congruence.
 It is worth mentioning that if this is the case then
the graded dimension of any $C_2$-cofinite, holomorphic vertex
operator algebra over $\mathbb C$ is a congruence modular function
(cf. \cite{DLM00} and \cite[Section 4]{kohnen-mason08}).

In this paper, we will restrict ourselves to a class of
noncongruence subgroups, called noncongruence character groups,
which are closely related to congruence subgroups. A group $\G$ is
called a character group of another finite index group $\G^0$ of the
modular group if $\G$ is normal in $\G^0$ with finite abelian
quotient.  By the definition, there is a surjective homomorphism
\begin{equation}
  \phi: \G^0 \twoheadrightarrow G
\end{equation} such that $\G=\ker \phi$ for some finite abelian group $G$. Note that   there exists another surjective homomorphism $\pi:
\G^0 \twoheadrightarrow H_1(X_{\G^0}, \ZZ)$, the first homology
group of  the compactified modular curve $X_{\G^0}$
 for $\G^0$ (cf. \cite[Prop. 1.6]{Manin72}).  In \cite{Kurth-Long06}, we distinguish two types of character
groups based on the level structures. A character group $\G$ of
$\G^0$ is said to be of \emph{type II } if the modular curve for
$\G$ is a finite covering of $X_{\G^0}$ unramified at the cusps.
Otherwise the group $\G$ is said to be of \emph{type I}. In
particular, a character group $\G$ of $\G^0$ is said to be of
\emph{type II(A)} if $\phi$ factors through the kernel of $\pi$; and
is said to be be \emph{type I(A)}, if the modular curve for $\G$ is
a finite covering of $X_{\G^0}$ unramified outside of the cusps of
$X_{\G^0}$ and $\pi(\ker \phi)= H_1(X_{\G^0},\ZZ)$. In the case that
$\G$ is a type I(A) character group of $\G^0$ with $\G^0/\G$
isomorphic to $\ZZ/n\ZZ$, the field of meromorphic modular functions
for $\G$ is a cyclic field extension of that for $\G^0$ which can be
generated by $\sqrt[n]{f}$ for some modular function $f$ for $\G^0$
whose zeros and poles are located at the cusps of $\G^0$. A modular
function $f$ with zeros and poles only at the cusps is called a
modular unit (cf. \cite{Kubert-Lang-81}).

A noncongruence subgroup $\G$ is said to satisfy the
\textbf{condition (UBD)} if the following conditions hold:
\begin{quote}
  \emph{If $f$ is an integral weight modular form for
        $\Gamma$    such that
  \begin{enumerate}
  \item $f$ is holomorphic on the upper half plane with poles only  at the
  cusps;
    \item $f$ has algebraic Fourier coefficients at infinity;
    \item $f$ is not a modular form for  $\Gamma^c$, the congruence
    closure of $\G$ in $\SLZ$,
  \end{enumerate}
  then $f$ has \textbf{unbounded denominators}, i.e. there is no algebraic
  integer $A\neq 0$ such that $A \cdot f$ has algebraic integer
  coefficients at infinity.}
\end{quote}
 It is conjectured that every noncongruence group
satisfies (UBD). If the conjecture is true, it provides a clear and
nice criterion for identifying which modular forms with algebraic
coefficients are congruence.

In this short note we prove the following  two results using a
similar argument which is derived from our previous discussion in
\cite{Kurth-Long06}:
\begin{theorem}\label{thm:main}
Let $\G^0=\G_0(M)$ with $M$  a square-free positive integer whose
genus is at  least 1. Then every type I(A) noncongruence character
group $\G$ of $\G_0(M)$ satisfies the condition (UBD).
\end{theorem}

\begin{theorem}\label{thm:main2}
  Let $\G^0$ be a genus 1 congruence subgroup whose modular curve has no complex multiplication.
  Then there exists an integer $M(\G^0)$ depending on $\G^0$ such that for any positive integer $n$ relatively prime to $M(\G^0)$, every index-$n$ type II(A) character group
  of $\G^0$ satisfies the condition (UBD).
\end{theorem}
This result overrides  Theorem 3 of \cite{Kurth-Long06} when the
modular curve  $X_{\G^0}$ has no complex multiplication. For
instance,  $X_{\G^0(11)}$ has no complex multiplication. Cummins and
Pauli have classified all congruence subgroups up to genus 25
\cite{cummins-pauli03}. Using their database together with a
computational package like MAGMA one can check explicitly which
genus 1 congruence subgroups satisfy the condition of Theorem
\ref{thm:main2}.

 As a corollary, we will prove Conjecture 37 of
\cite{Kurth-Long06}.
\begin{theorem}\label{thm:main3}Every noncongruence type II character group
of $\G^0(11)$ satisfies the condition (UBD).
\end{theorem}

Note that a modular form for a character group is automatically a
generalized modular form (GMF)
 (cf. \cite{knopp-mason99-JNT} and also the definition in
\cite{kohnen-mason08}). Kohnen and Mason pointed out to the second
author that many GMF's have unbounded denominators. They obtained
several results in \cite{kohnen-mason08} regarding the coefficients
of GMF's with empty or cuspidal divisor.

 In the
appendix, we show that if the expansion of a modular form at one
cusp has algebraic coefficients, then so does its expansion at any
other cusp. It is a fact well-known to the experts and is used in
the proofs, but since we could not find a proof in the literature we
provide one here for the sake of completeness.

For convenience, we say a Laurent power series satisfies the
condition (\textbf{FS-AB}) if its coefficients are algebraic and
have bounded denominators.

\section{
Type I(A) character subgroups of $\G_0(M)$ with $M$
square-free}\label{sec:typeI}
\begin{lemma}\label{lem:3}
  If $\G$ is a type I(A) (resp. II(A)) character group of $\G^0$, then $\G=\bigcap_{i=1}^s \G_i$ such that each $\G_i$ is a type I(A) (resp. II(A)) character
  group of $\G^0$ with $\G^0/\G_i\cong \ZZ/p_i^{e_i} \ZZ$ for some primes $p_i$ and positive integers
  $e_i$.
\end{lemma}
\begin{proof}By the definition of character groups, there is a surjective homomorphism $\phi: \G^0
\rightarrow G$ where $G$ is a finite abelian group such that
$\G=\ker \phi$.
  By the Fundamental Theorem of Finite Abelian Groups, ${\displaystyle G\cong
  \bigoplus_{i=1}^s \ZZ/p_i^{e_i} \ZZ}$.
  Let $\phi_i: \G^0 \rightarrow \ZZ/p_i^{e_i} \ZZ$ be the natural
  projections
  of $\phi$ and $\G_i=\ker \phi_i$. Therefore, $\G=\bigcap _{i=1}^s \G_i$. If
  $\G$ is a type I(A) (resp. II(A)) character group of $\G^0$ then by the
  definitions  each $\G_i$ as a character group of $\G^0$  is also of  type I(A) (resp. II(A)).
\end{proof}

Let $\frak M_{\G}$ (resp. $\frak M_{\G^0}$) denote the
field of meromorphic modular functions for $\G$ (resp. $\G^0$). Let
 $ c_1,\cdots, c_{t-1}, c_t=\infty$ be the list
of  cusps of $X_{\G^0}$ with $\g_i$ the generator of the stabilizer
of $c_i$. By the Manin-Drinfeld theorem, $(c_i)-(\infty)$ is an
order $N_i$ torsion point of the Jacobian $J_0(M)$ of $X_{\G_0(M)}$.
Therefore, there is a modular function $h_i\in \frak M_{\G_0(M)}$
such that $\text{div}(h_i)=N_i((c_i)-(\infty))$ for every
$i=1,\cdots,t-1$. For a fixed prime number $p$, the extension $\frak
M_{\G^0}(\sqrt[p]{h_i})$ over $\frak M_{\G^0}$ corresponds to a type
I(A) character group of $\G^0$ if, and only if, $N_i$ is relatively
prime to $p$.
\begin{lemma}\label{lem:I(A)chargp}
\begin{enumerate}
  \item If $\G$ is a type I(A) character group of $\G^0$, then any
  intermediate group sitting between $\G^0$ and $\G$ is also a type
  I(A) character group of $\G^0$.
  \item The intersection of two type I(A) character groups of $\G^0$
  is also a type I(A) character group of $\G^0$.
  \item  For any integer $e\ge 1$, the extension $\frak
  M_{\G^0}(\sqrt[p^e]{h_i})$ over $\frak
  M_{\G^0}$
corresponds to a type I(A) character group of $\G^0$ if, and only
if, $N_i$ is relatively prime to $p$.
\item There are $p^{e(t-2)}+p^{e(t-3)}+\cdots +1$ non-isomorphic
index-$p^e$ type I(A) character groups of $\G^0$ whose field
extensions over $\frak M_{\G^0}$ can be generated by modular units.
\end{enumerate}
\end{lemma}
\begin{proof}
By the definition of type I(A) character groups, it is
straightforward to verify the first two claims.

 By part (1), we know that for any integer $e\ge 1$,
$\frak M_{\G^0}(\sqrt[p^e]{h_i})$ corresponds to a type I(A)
character group of $\G^0$ if and only if $N_i$ is relatively prime
to $p$. By part (2), $\sqrt[p^e]{h_1^{a_1}\cdots h_{t-1}^{a_{t-1}}}$
with   integer $a_i$'s gives rise to a type I(A) character group if
and only if $p$ is relatively prime to every $N_i$. Since each $h_i$
is a modular unit, so is $\sqrt[p^e]{h_1^{a_1}\cdots
h_{t-1}^{a_{t-1}}}$. Treating $(a_1, \dots, a_{t-1})$ as an element
of $\PP^{t-2}(\ZZ/p^e \ZZ)$, there is one non-isomorphic index-$p^e$
type I(A) character group of $\G^0$, for each element of
$\PP^{t-2}(\ZZ/p^e \ZZ)$, namely $\frac{p^{e(t-1)}-1}{p^e-1}$ such
groups up to isomorphism.

\end{proof}

\begin{lemma}
If $\G$ is a type I(A) character group of $\G^0$ with $\G^0/\G\cong
\ZZ/p^e \ZZ$ for some prime power $p^e$, then $\frak M_{\G}=\frak
M_{\G^0} (\sqrt[p^e]{f})$ for some modular unit $f$ in $\frak
M_{\G^0}$.
\end{lemma}

\begin{proof}
  By the Galois correspondence, $\frak M_{\G}$ is a finite Galois
  extension of $\frak M_{\G^0}$ whose Galois group is isomorphic to
  $\G^0/\G$. Hence $\frak M_{\G}=\frak
M_{\G^0} (\sqrt[p^e]{f})$ for some $f\in \frak M_{\G^0}$.

Now we  show that $f$ can be chosen as a modular unit.  If
$\phi:\G^0 \rightarrow \ZZ/p^e\ZZ$ is a group homomorphism such that
$\G=\ker\phi$ is of type I(A), then $\phi$ is completely determined
by the parabolic elements $\g_1,\cdots, \g_{t-1}$ and $N_i$ has to
be relatively prime to the order of $\phi(\g_i)$  for every
$i=1,\cdots, t-1$.  By a counting argument, we know there are
$\frac{p^{e(t-1)}-1}{p^e-1}$ non-isomorphic index-$p^e$ type I(A)
character groups of $\G^0$ with cyclic quotient. By part (4) of the
previous lemma, this is the same number as arise from modular units,
proving the claim.
\end{proof}
Recall that $\eta(z)=q^{1/24}\prod_{n\ge 1} (1-q^n), \ q=e^{2\pi i
z}$ is the classical Dedekind eta function. Below, we call a
function $f$  an eta quotient if $f=\prod_{j=1}^t \eta(a_j z)^{e_j}$
for $a_j\in \mathbb N$ listed in a strictly increasing order and
$e_j\in \ZZ\setminus \{0\}$.

\begin{theorem}[Tagaki \cite{Takagi97}]\label{thm:tagaki}
Up to a scalar multiple, every modular unit for $\G_0(M)$ with the
positive integer $M$ square-free is an eta quotient.
\end{theorem}
The following lemma is a special case of Lemma 11 in
\cite{Kurth-Long06}. For any $n\ge 1$ and with a principal branch
fixed, we formally write
\begin{equation}\label{eq:nthroot}
 (1+x)^{1/n}= \sqrt[n]{1+x}=\sum_{m\ge 0}\frac{\left (\frac{1}{n} \right )_m}{m!}x^m,
\end{equation} where
$(\frac{1}{n})_m=\frac{1}{n}(\frac{1}{n}-1)\cdots
(\frac{1}{n}-m+1)$.
\begin{lemma}\label{lem:ubd-h}
 Let $n$ be any natural
 number and $f=1+\sum_{m\ge 1}a(m)w^m, a(m)\in \ZZ$ for all $m$.  In terms of (\ref{eq:nthroot}),  we expand
$\sqrt[n]{f}=\sum_{m\ge 0} b(m)w^m,$ with  $b(m)\in \ZZ[1/n]$
formally. Let $p$ be a prime factor of $n$. If there exists one
 $b(m)$ which is not $p$-integral, then
$$\limsup_{m\rightarrow \infty} \left (-\ord_p b(m) \right ) \rightarrow \infty.$$
In other words, $\{b(m)\}$ has unbounded denominators.
\end{lemma}

\begin{lemma}\label{lem:etaubd}
  Let $f=\prod_{j=1}^t \eta(a_j z)^{e_j}$ be an eta quotient. For any prime power $p^e$ not dividing the
greatest common divisor of the $e_j$'s, the Fourier coefficients of
$\sqrt[p^e]{f}$ have unbounded denominators.
\end{lemma}
\begin{proof}
We may assume the greatest common divisor of the exponent $e_j$'s is
1. We now show that the coefficients of $\sqrt[p]{f}$ have unbounded
denominators. If not, by Lemma \ref{lem:ubd-h} the expansion
$\sqrt[p]{\prod_{j=1}^t \eta(a_j z)^{e_j}}=\sum b(n)q^{n/p}$
satisfies $b(n)\in \ZZ$. By Proposition 2.1 of
\cite{Bruinier-Kohnen-Ono04} we can write $\sum b(n)q^{n/p}$
uniquely into the form $q^{r}\prod_{n\ge 1} (1-q^n)^{c(n)}$ for some
rational number $r$ and complex numbers $c(n)$'s which can be
determined by the $b(n)$'s recursively. It is straightforward to
check that if the $b(n)$'s are all integers then so are the
$c(n)$'s. On the other hand it is easy to rewrite
$\sqrt[p]{\prod_{j=1}^t \eta(a_j z)^{e_j}}$ into the infinite
product form $q^{r'}\prod_{n\ge 1} (1-q^n)^{c'(n)}$ directly. If
$n_0$ is the least positive integer such that $p\nmid e_{n_0}$ then
$c'(a_{n_0})$ is not an integer. By the uniqueness of the  $c(n)$'s,
$c(n)=c'(n)$ which leads to a contradiction.
\end{proof}

\begin{cor}
Let $f=\prod_{j=1}^t \eta(a_j z)^{e_j}$ be an eta quotient, $d$ be
the greatest common divisor of the $e_j$'s.  If a prime $p\nmid d$,
then the Fourier coefficients of $\sqrt[p]{f}$ at infinity have
unbounded denominators, and so do the Fourier coefficients at any
cusp of $\G$. The modular function $\sqrt[p]{f}$ is modular for a
congruence subgroup if and only if $p$ divides $d$.
\end{cor}
\begin{proof}Let $c$ be a cusp and $\gamma_c\in \SLZ$ such that $\gamma_c
  \infty =c$. The expansion of any modular form $h$ at $c$ is the expansion of $h|_{\gamma_c}$ at infinity. Since $f$ is an eta quotient, $f|_{\gamma_c}$ is also an eta
quotient by the well-known transformation formulae of the eta
function.

If $p\mid d$, $\sqrt[p]{f}$ is an eta quotient and hence congruence,
then so is $\sqrt[p]{f}|_{\gamma}$ for
  any $\g\in \SLZ$. Thus the Fourier expansion of $f$ at any cusp
  satisfies (FS-AB). Conversely, if $\sqrt[p]{f}|_{\gamma}$
  satisfies (FS-AB) for some $\gamma$, then $\sqrt[p]{f}|_{\gamma}$
  is an eta quotient and hence congruence. This implies $\sqrt[p]{f}$
  is also congruence. Therefore it is an eta quotient and $p\mid d$.
\end{proof}

\begin{lemma}\label{lem:A2} If $g(z)$ is a modular function of a congruence group
 with poles only at the cusps and algebraic Fourier
coefficients, then there is a constant $A$ such that $A \cdot g(z)$
has algebraic integer Fourier coefficients.
\end{lemma}
\begin{proof}
Let $\Delta(z) = \eta^{24}(z)$ which  is a cuspform for $\SLZ$ with
series:
\begin{equation*}
  \Delta(z) = q -24 q^2 + 252 q^3 +\dots
\end{equation*}
In particular, the Fourier coefficients are all integers.
Multiplying $g$ by powers of $\Delta$ will kill the poles at the
cusps, hence $\Delta^n g$ is also a cuspform for sufficiently large
$n$, and there is a constant $A$ such that $A \cdot \Delta^n g$ has
algebraic integer Fourier coefficients (as a result of Theorem 3.52
in \cite{shim1}). But $\frac{1}{\Delta}$ has algebraic integer
Fourier coefficients as well, since
\begin{equation*}
  \frac{1}{\Delta} = \frac{1}{q} \cdot \frac{1}{1+(\Delta/q-1)}
   = \frac{1}{q}\left(1 - \left(\frac{\Delta}{q}-1\right) + \left(\frac{\Delta}{q}-1\right)^2 -
   \dots\right)
\end{equation*}
So $A \cdot g = A \Delta^{-n}\Delta^n g$ has algebraic integer
Fourier coefficients.
\end{proof}

We are ready to prove  Theorem \ref{thm:main}.
\begin{proof}[Proof of Theorem \ref{thm:main}]Let $\G$ be a type I(A) noncongruence character group of $\G_0(M)$.
Assume $\G_0(M)/\G\cong \bigoplus_{i=1}^s \ZZ/p_i^{e_i}\ZZ$ and
$\G\cong \bigcap_{i=1}^s \Gamma_ i$ where $\G_i$ are type I(A)
character groups of $\G_0(M)$ with $\G_0(M)/\G_i\cong \langle
\gamma_i \G \rangle \cong \ZZ/p_i^{e_i}\ZZ$ for some $\g_i\in
\G_0(M)$ and certain prime powers $p_i^{e_i}$. We further assume
that each $\frak M_{\G_i}$ is generated over $\frak M_{\G_0(M)}$ by
$g_i=\sqrt[p^{e_i}]{f_i}$ with $f_i$ being a modular unit for
$\G_0(M)$. By Theorem \ref{thm:tagaki}, we can assume that each
$f_i$ is an eta quotient. Consequently, a basis of $\frak M_{\G}$
over $\frak M_{\G_0(M)}$ is $S=\left \{\prod_{i=1}^s g_i^{n_i}
\right \}_{0\le n_i\le
 p^{e_i}-1}.$

Let $h$ be an integral weight $k$  modular form for $\G$ holomorphic
on the upper half plane and satisfying (FS-AB). Up to multiplying
with a suitable newform for $\G_0(M)$  one can assume $k$ is a
multiple of 12. Dividing by $\Delta^{k/12}$ we obtain a modular
function for $\G$ satisfying (FS-AB). From now on we assume that $h$
is of weight 0. The goal is to show such a modular function $h$,
holomorphic on the upper half plane and satisfying (FS-AB), must be
congruence.

Write $h=\sum_{I=(n_1,\cdots, n_s)} a_{I}\prod_{i=1}^s g_i^{n_i}$
with $a_I\in \frak M_{\G_0(M)}$. For convenience, we denote
$\prod_{i=1}^s g_i^{n_i}$ by $g^I$. Note that $g_i|_{\g_j}=g_i$ if
$i\neq j$ and $g_i|_{\g_i}=\mu_{p^e}g_i$ where $\mu_n$ stands for a
primitive $n$th root of unity.  So for every $\gamma \G \in
\G_0(M)/\G$, $g^I|_{ \gamma}=\phi_I(\gamma)g^I$ for some character
$\phi_I: \G_0(M)/\G \rightarrow \CC^{\times}$ of $\G_0(M)/\G$. The
$\phi_I$'s are non-isomorphic and they form the complete set of
non-isomorphic characters of the abelian quotient group
$\G_0(M)/\G$. Hence each $a_{I}g^I$ is a linear combination of
$h|_{\g_1^{n_1}\cdots \g_s^{n_s}}$ with $n_i\in \{0,\cdots,
p^{e_i}-1\}$ and some scalars in a cyclotomic field. Note that each
$h|_{\g_1^{n_1}\cdots \g_s^{n_s}}$ is also holomorphic on the upper
half plane with algebraic coefficients (cf. Appendix), thus so is
each $a_{I}g^I$. Also each $g_i$ is nonzero in the upper half plane,
so each modular function $a_I\in \frak M_{\G_0(M)}$ is also
holomorphic on the upper half plane with algebraic coefficients. By
Lemma \ref{lem:A2}, each $a_I$ satisfies (FS-AB).

We partition the basis $S$ into two sets  $S_c$ and $S_n$. An
element in $S$ belongs to $S_c$ if it is congruence
 and otherwise  it belongs to $S_n$. Note that $(h)_c=\sum_{I\in S_c} a_{I}g^I$ is a congruence modular form  which is holomorphic on the
upper half plane, hence it satisfies (FS-AB).  So $(h)_n=\sum_{I\in
S_n} a_{I}g^I=0$ also satisfies (FS-AB).

If  there are $g^I$ and $g^{I'}$ in $S_n$ such that $g^{I'}/g^I=E$
is an eta
 product, then $a_Ig^I+a_I'g^{I'}=(a_I+a_I'E)g^I$ with $a_I+a_I'E$ being congruence and
satisfying (FS-AB). Hence one can further assume that for every two
elements in
 $S_n$ their quotient is not a congruence modular form.

With the  assumptions above, let $M((h)_n)$ be the number of nonzero
$a_I$'s in the expression of $(h)_n$. We will conclude $M((h)_n)=0$
by using an argument similar to the proof of Lemma 13 in
\cite{Kurth-Long06} to exclude the remaining possibilities.

\emph{Case 1}: $M((h)_n)=1$. In this case $a_Ig^I$ satisfies (FS-AB)
for some nonzero $a_I$ satisfying (FS-AB). Since
$(g^I)^{|\G_0(M)/\G|}$ is an eta quotient, it satisfies (FS-AB). By
Lemma 39 of \cite{Kurth-Long06}, $(g^I)^{1+|\G_0(M)/\G|}$ satisfies
(FS-AB). Note that the reciprocal of the eta quotient
$(g^I)^{|\G_0(M)/\G|}$ satisfies (FS-AB). It follows $g^I$ also
satisfies (FS-AB). By Lemma \ref{lem:etaubd}, $g^I$ is congruence.
This contradicts our assumption on $(h)_n$.

\emph{Case 2}: $M((h)_n)>1$. Let $\mathcal D$ be the differential
operator defined in the proof of Lemma 13 in \cite{Kurth-Long06}. If
$h$ is a formal power series whose coefficients have bounded
denominators, then so is $\mathcal D (h)$. Following the argument of
the proof of Lemma 13 in \cite{Kurth-Long06}, there exists a nonzero
modular function $b_I$ for $\G_0(M)$ holomorphic on the upper half
plane satisfying (FS-AB) such that $\widetilde{h}=(b_I-a_I\mathcal
D) ((h)_n)\neq 0$
  and  $(\widetilde{h})_n= \widetilde{h}$.
Moreover, $M((h)_n)>M(\widetilde{h})$. By induction, this case
reduces to back to case  1.

So $h=(h)_c$ is congruence.
\end{proof}

In this proof, if we  replace $h$ by $h|_{\gamma}$ for any $\g\in
\SLZ$ then each $h|_{\gamma}$ is also a combination of $n$th roots
of eta quotients whose coefficients are congruence modular forms
holomorphic on the upper half plane. Consequently, one can
strengthen the (UBD) condition in this case to: for every genuine
noncongruence modular form, holomorphic on the upper-half plane with
algebraic coefficients, its Fourier expansion at \emph{every cusp}
has \emph{unbounded denominators}.

\section{Type II(A) character groups of genus 1 congruence subgroups}
In this section, we follow closely \cite{Kurth-Long06} and the
approach in the previous section.  Let $\G^0$ be a genus 1
congruence subgroup whose modular curve $X_{\G^0}$  is defined over
a number field $K$ and has no complex multiplication. By the theory
of elliptic functions, there exist two modular functions $x$ and $y$
for $\G^0$ with poles of order 2 and 3 respectively at infinity  and
holomorphic everywhere else. The modular functions $x$ and $y$
satisfy $y^2=x^3+Ax+B$ for some  $A,B\in K$. Moreover, the Fourier
coefficients of $x=w^{-2}+a_{-1}w^{-1}+\cdots $ and
$y=w^{-3}+b_{-2}w^{-2}+\cdots, w=e^{2\pi i/\mu}$ are in $K$, where
$\mu$ is the cusp width of $\G^0$ at infinity. By Lemma
\ref{lem:A2}, $x$ has bounded denominators and there exists a
rational integer $N(\G^0)$ depending on $\G^0$ such that for all
prime ideals $\wp$ of $\mathcal O_K$ not dividing $N(\G^0)$, the
coefficients of $x$ are all $\wp$-integral. Let $R=\Z[A,B]$. By
\cite[Ex. 3.7 pp. 105]{sil1},
  there exists a polynomial
  \begin{equation}\label{psi}
\psi_{p}(x)=px^{(p^2-1)/2}+c_{(p^2-1)/2-1} x^{(p^2-1)/2-1}+\cdots
+c_1x+c_0 \in
  R[x]
  \end{equation} satisfied by the $x$-coordinates of the order-$p$
  points of $X_{\G^0}$. Since  $X_{\G^0}[p]$ over $\F_p$ is isomorphic to either
  $\{ 0\}$ or $\Z/p\Z$ (cf. \cite[Theorem 3.1]{sil1}), $p\nmid c_n$
  for some $n$. It follows that there exists one $p$-torsion point $P_0$, whose $x$-coordinate
  is not algebraically integral over $\wp_0$ for some prime ideal above
  $p$.  By a
 result of Serre \cite{ser76}, the homomorphism
\begin{equation}\label{eq:varphi_p}
\varphi_p: \Gal(\overline{\Q}/\Q) \rightarrow
\text{Aut}(X_{\G^0}[p])\cong GL_2(\F_p)
\end{equation}
  on
  the $p$-torsion points of $X_{\G^0}$ is
  surjective for almost all primes $p$ when $X_{\G^0}$ has no complex multiplication.
  When
  $\varphi_p$ is surjective,
  $\Gal(\overline{\Q}/\Q)$ acts on $X_{\G^0}[p]$
  transitively. Consequently for any   $P\in X_{\G^0}[p]$,
  $x(P)$ is not algebraically integral over some prime $\wp$ above $p$.

\begin{lemma}
  If $\a$ is an algebraic number which is not $\wp$-integral for some prime ideal $\wp\nmid N(\G^0)$, then the Laurent power series  $(x-\alpha)^{-1}$ in $w$ has unbounded denominators.
\end{lemma}
\begin{proof}Assume $x=w^{-2}+a_{-1}w^{-1}+a_0+\cdots$.
It is equivalent to show that
$(1+a_{-1}w+(a_0-\a)w^2+\cdots)^{-1}=1+\beta+\beta^2+\cdots=1+\sum
c(n)w^n$ has unbounded denominators where
$\beta=-(a_{-1}w+(a_0-\a)w^2+\cdots)$. It is straightforward to
verify that if $\ord_{\wp}\a=-r$, then $\ord_{\wp}c(2n)=-nr$. So
$(x-\alpha)^{-1}$ has unbounded denominators $\wp$-adically.
\end{proof}

Given a $p$-torsion point $P$ of $X_{\G^0}$, let $f_P\in \frak
M_{\G^0}$ be a modular function whose divisor satisfies that
$\text{div}{f_P}=p(\infty)-p(P)$. We can assume that the
coefficients of $f_P$ are algebraic (\cite[Lemma 23]{Kurth-Long06}).
We choose $f_{-P}$ in a similar way.

\begin{lemma}Let $p$ be a prime not dividing $N(\G^0)$,  $f_P$ and
$f_{-P}$ as above.
  Then at least one of  $(f_P)^{1/p}$ or $(f_{-P})^{1/p}$
  has unbounded denominators.
\end{lemma}
\begin{proof}
By checking the divisors we know that
$(f_Pf_{-P})^{1/p}=(x-x(P))^{-1}$ up to a scalar. By the previous
lemma $(f_Pf_{-P})^{1/p}$ has  unbounded denominators. Thus, at
least one of $(f_P)^{1/p}$ or $(f_{-P})^{1/p}$
  satisfies the unbounded denominator property.
\end{proof}
Without loss of generality we assume that $g_P=\sqrt[p]{f_P}$ has
unbounded denominators $\wp$-adically for some prime $\wp$ above
$p$. So does $(g_P)^{j}$ for any integer $j$ which is relatively
prime to $p$. Therefore,
\begin{lemma}\label{lem:g_P^j}Under the above assumptions,
$(g_P)^{j}$ does not satisfies (FS-AB) for any integer
$j\in\{p+1,p+2,\cdots, 2p-1\}$,
\end{lemma}

\begin{theorem}
  Let $\G^0$ be a genus 1 congruence subgroup whose modular curve has no complex multiplication. Then for almost all primes $p$, every index-$p$ type II(A) character group
  of $\G^0$ satisfies the condition (UBD).
\end{theorem}

\begin{proof} Let $p$ be a prime which is relatively prime to $N(\G^0)$ and such that the homomorphism $\varphi_p$ \eqref{eq:varphi_p} is
surjective. Let $\G$ be an index-$p$ type II(A) character group of
$\G^0$.   From \cite[Proposition 25]{Kurth-Long06}, $\frak
M_{\G}=\frak M_{\G^0}(g_P)$ for some $g_P$ as above. We will show
that such a group $\G$ satisfies the condition (UBD). If not, one
can construct a genuine noncongruence modular function $f\in \frak
M_{\G}$ which is holomorphic on the upper half plane and satisfies
(FS-AB) by Lemma \ref{lem:A2}. We can write $f=\sum_{j=0}^{p-1} a_j
g_P^j, a_j\in \frak M_{\G^0}$. Assume $\G^0/\G= \langle \gamma \G
\rangle$. Then $g_P|_{\gamma}=e^{2\pi i/p} g_P$. Since $f|_{\g}$ is
also holomorphic on the upper half plane, so is every $a_j g_P^j$
which is a combination of $f|_{\g^j}$'s. So the poles of the
congruence modular functions $a_j$ are supported at the cusps. Thus
each $a_j$ satisfies (FS-AB) by Lemma \ref{lem:A2}. By Lemma 13 of
\cite{Kurth-Long06} (and the proof of Theorem \ref{thm:main}), for
some $j\in \{p+1,\cdots, 2p-1\}$ ${g_P}^{j}$ satisfies (FS-AB)
 which contradicts Lemma
\ref{lem:g_P^j}.
\end{proof}

\begin{proof}[Proof of Theorem \ref{thm:main2}]
Let $M(\G^0)$ be the  product of  $N(\G^0)$  and all primes $p$ such
that $\phi_p$ is not surjective. Now let $\G$ be an index-$n$ type
II(A) character group of $\G^0$ such that $(n,M(\G^0))=1$. By Lemma
\ref{lem:3}, $\G=\bigcap_{i=1}^s \G_i$ where each $\G_i$ is a type
II(A) character group of $\G^0$ with $\G/\G^i\cong \Z/p^{e_i}\Z$ for
some prime power $p^{e_i}>1$ relatively prime to $M(\G^0)$.  Because
$\frak M_{\G_i}$ is a cyclic extension over $\frak M_{\G^0}$ of
order $p^{e_i}$, it is generated by some modular function $g_i$. Let
$G_i$ be the unique index-$p$ subgroup of $\G^0$ which contains
$\G_i$. By the proof of the previous theorem, $\frak M_{G_i}=\frak
M_{\G^0}(g_P)$ for some modular function $g_P$ as before. Moreover,
we can assume that $g_P$ has unbounded denominators and
$g_i^{p^{e_i-1}}=g_P$. It follows that $g_i$ has unbounded
denominators too.

Like the case in Section \ref{sec:typeI}, the set of modular
functions $S=\left \{\prod_{i=1}^s g_i^{n_i} \right \}_{0\le n_i\le
 p^{e_i}-1}$ is  a basis of $\frak M_{\G}$ over $\frak
 M_{\G^0}$. If $h$ is a modular function for $\G^0$ which is
 holomorphic on the upper half plane and satisfies (FS-AB), then
 following the argument of the proof of Theorem \ref{thm:main} we
 know $h$ must be congruence. (Under our assumption on $n$, $S_n=S$ in this
 case.) This implies  the claim of Theorem \ref{thm:main2}.
\end{proof}

\begin{proof}[Proof of Theorem \ref{thm:main3}]Since the modular curve for
$\G^0(11)$ has no elliptic points, a type II character group $\G$ of
$\G^0(11)$ is   automatically  of type II(A).

We now show that $M(\G^0(11))=5$. By a result of Cojocaru
\cite{Cojocaru05}, when $p>37$, $\phi_p$ is surjective for the
elliptic curve $X_{\G^0(11)}$. Thus it boils down to checking that
the polynomial $\psi_p(x)$ (cf. \eqref{psi}) is irreducible over
$\Q$ when $p\le 37$ and $p\neq 5$, which can be done
computationally. Therefore, Theorem \ref{thm:main2} and
\cite[Theorem 36]{Kurth-Long06} imply that every type II
noncongruence subgroup of $\G^0(11)$ satisfies the condition (UBD),
which is equivalent to the claim of \cite[Conjecture
37]{Kurth-Long06}.


\end{proof}

\section{Appendix}
The goal of this appendix is to show the following proposition used
in the previous proof.
\begin{prop}\label{cor_to_ctok}
Let $f$ be a modular function for $\Gamma$ whose Fourier expansion
about $\infty$ has coefficients in a number field $K$. Then for
every $\gamma \in \SLZ$, the Fourier expansion of $f|_\gamma$ about
$\infty$, or the expansion of $f$ at the cusp $\gamma \cdot \infty$,
also has coefficients in number field $K'$ ($K'$ may be larger than
$K$ in general).
\end{prop}

Let $j(z)$ be the classical modular $j$-function. By Theorem 1 of
\cite{a-sd} there is an irreducible polynomial $g(x,y) \in \CC[x,y]$
such that $g(f,j) = 0$. Since both $f$ and $j$ have algebraic
coefficients at infinity, one can use an elementary argument to show
that up to a scalar $g(x,y) \in K' [x,y]$ for some number field
$K'$. Since $j|_{\gamma}=j$ for all $\gamma\in \SLZ$,
$g(f|_{\gamma},j)=0$ for all $\gamma \in \SLZ$. The claim of the
proposition is equivalent to saying every solution of $g(f,j)=0$, as
a formal power series, has algebraic Fourier coefficients. For now
on, we use $g(f,q)$ to denote a polynomial in variable $f$ with
coefficients in the ring of Laurent series in $q$.

\begin{lemma}
Let $M$ be a nonnegative integer. Then:
\begin{equation*}
 \frac{d^M}{dq^M}\left( g(f(q),q)\right) =
 \sum_p\left( c_p\left( \frac{\partial^{n+M-d} g(f,q)}{\partial f^n \partial q^{M-d}}\right)
 \cdot \prod_{i=1}^n \frac{d^{d_i} f}{d q^{d_i}} \right)
\end{equation*}
where the sum ranges over all partitions
\begin{equation*}
p: d_1 + \dots + d_n = d, \ d_i\ge 1
\end{equation*}
of all $d \in [0,M]$ (where the partition of $0$ is empty), and
$c_p$ is a combinatorial constant:
\begin{equation*}
  c_p = \binom{M}{d} \frac{d!}{d_1 ! \cdots d_n!} \cdot
  \frac{1}{\prod_{i=1}^d \#(i\in p)!}
\end{equation*}
\end{lemma}
\begin{proof}

First note that:
\begin{equation*}
\frac{d}{dq}\left( \frac{\partial^{a+b} g(f,q)}{\partial f^a
\partial q^b}\right) = \frac{\partial^{a+b+1} g}{\partial f^{a+1}
\partial q^b} \frac{df}{dq} + \frac{\partial^{a+b+1} g}{\partial f^a
\partial q^{b+1}}
\end{equation*}
We claim that every term of $\displaystyle \frac{d^m}{dq^m}
(g(f(q),q))$ is of the form:
\begin{equation*}
\frac{\partial^{n+m-d} g}{\partial f^n \partial q^{m-d}} \cdot
\prod_{i=1}^n \frac{d^{d_i} f}{d q^{d_i}}
\end{equation*}
where $d_1 + \dots + d_n = d \leq m$.

Suppose this is true for $m$. Then $\displaystyle \frac{d}{dq}\left(
\frac{\partial^{n+m-d} g}{\partial f^n
\partial q^{m-d}} \cdot \prod_{i=1}^n \frac{d^{d_i} f}{d q^{d_i}} \right)$
has three types of terms, corresponding to new partitions:
\begin{eqnarray*}
 \frac{\partial^{n+m-d+1} g}{\partial f^{n+1} \partial q^{m-d}} \cdot
 \prod_{i=1}^{n+1} \frac{d^{d_i} f}{d q^{d_i}}
  &\qquad \text{for $p: d_1 + \dots + d_n + 1$}\\
 \frac{\partial^{n+m-d+1} g}{\partial f^n \partial q^{m-d+1}} \cdot
 \prod_{i=1}^n \frac{d^{d_i} f}{d q^{d_i}}
  &\qquad \text{for $p: d_1 + \dots + d_n$}\\
 \frac{\partial^{n+m-d} g}{\partial f^n \partial q^{m-d}} \cdot
 \prod_{i=1}^n \frac{d^{d_i} f}{d q^{d_i}}
  &\qquad \text{for $p: d_1 + \dots + (d_j + 1) + \dots + d_n$}\\
\end{eqnarray*}
where for the last type, there is one for each $1\leq j \leq n$.

Thus differentiation on a term corresponding to a partition $p$
splits $p$ up into $n+2$ partitions: $p$ itself, $p$ appending
``+1'', and all terms $p$ with $1$ added to one of the elements of
$p$.

To get the combinatorial coefficient, we count how many ways to get
to a partition $p$ in $M$ steps using the three rules above. If
$M>d$ there are steps where $p$ doesn't change, and they can be put
in any order, hence the $\binom{M}{d}$ term in $c_p$. The remaining
steps consist of adding $+1$ to the $d_i$'s, hence the multinomial
coefficient, and the remaining term is to remove any overlap in
counting when $d_i = d_j$ for some $i$ and $j$.
\end{proof}

Let $g(x,q)$ be a degree $N$ polynomial (in $x$) with coefficients
in $K[[q]]$ for some field $K$:
\begin{eqnarray*}
g(x,q) &=& \sum_{i=0}^N g_i(q) x^i=\sum_{j=P}^\infty h_j(x)q^j\,.
\end{eqnarray*}
We want to find $f(q)$ such that $g(f(q),q) = 0$. If the order of
$f(q)$ at $\infty$ is $Q$ then $q^{-Q} f(q)$ is holomorphic and
non-zero at $\infty$, and it satisfies $\overline{g}(q^{-Q}f(q),q) =
0$ where
\begin{equation*}
  \overline{g}(x,q) = \sum_{i=0}^N (g_i(q)\cdot q^{Qi}) x^i\,.
\end{equation*}
So, in solving for $f(q)$, we can adjust the $h_j$ polynomials and
assume $f$ is holomorphic and non-zero at $\infty$. Moreover, we can
assume $P = 0$ (and hence each $g_i(q)$ is holomorphic at $\infty$)
since we can multiply powers of $q$ to both sides of $g(f(q),q) =
0$. Let:
\begin{equation*}
f(q) = \sum_{i=0}^\infty a_i q^i\,.
\end{equation*}
We will plug these series into the Lemma. Note that:
\begin{equation*}
 \frac{\partial^{a+b}g(f,q)}{\partial f^a \partial q^b} =
 \sum_{j=b}^\infty h_j^{(a)}(f) j(j-1)\dots (j-b+1) q^{j-b}\,.
\end{equation*}
So:
\begin{equation*}
\left.\frac{\partial^{a+b}g(f,q)}{\partial f^a \partial
q^b}\right|_{q=0} = b! h_b^{(a)}(a_0)\,.
\end{equation*}
Similarly:
\begin{equation*}
  \left.\frac{d^b f}{dq^b}\right|_{q=0} = b! a_b\,.
\end{equation*}

Now let $Q_M$ be the $M$th coefficient of $g(f(q),q)$. Since
\begin{equation*}
  Q_M = \left. \frac{1}{M!} \frac{d^M}{dq^M}
  g(f(q),q)\right|_{q=0}\,,
\end{equation*}
putting it all together we have:
\begin{eqnarray*}
  Q_M &=&\frac{1}{M!} \sum_p\left. \left( c_p \frac{\partial^{n+M-d}
  g}{\partial f^n \partial q^{M-d}}
  \cdot \prod_{i=1}^n \frac{d^{d_i} f}{d q^{d_i}} \right)\right|_{q=0}\\
      &=&\frac{1}{M!} \sum_p\left( \binom{M}{d} \frac{d!}{d_1 ! \cdots d_n!} \cdot
  \frac{1}{\prod_{i=1}^d \#(i\in p)!} \cdot
  (M-d)! h_{M-d}^{(n)}(a_0) \prod_{i=1}^n d_i! a_{d_i} \right)\\
      &=& \sum_p \frac{1}{\prod_{i=1}^d \#(i\in p)!}
          h_{M-d}^{(n)}(a_0) \prod_{i=1}^n a_{d_i}\,.
\end{eqnarray*}

For example:
\begin{eqnarray*}
  Q_4 &=& a_4 h'_0(a_0) + a_3 a_1 h''_0(a_0) + a_3h'_1(a_0) + \frac{1}{2}
  a_2^2 h''_0(a_0) + \frac{1}{2}a_2 a_1^2 h'''_0(a_0) + \\
  && + a_2 a_1 h''_1(a_0) + a_2 h'_2(a_0) + \frac{1}{24} a_1^4
  h''''_0(a_0) + \frac{1}{6} a_1^3 h'''_1(a_0) +\\
  && + \frac{1}{2} a_1^2 h''_2(a_0) + a_1 h'_3(a_0) + h_4(a_0)\,.
\end{eqnarray*}

We solve $g(f(q),q)=0$ for the $a_i$'s. Since $Q_0 = h_0(a_0)$, we
pick $a_0$ to be any non-zero root of $h_0$, an (at most) $N$th
degree polynomial. If $a_0$ is a simple root we will see that we can
successively solve each $a_i$. Suppose however that $h_0^{(i)}(a_0)
= 0$ for all $i \in \{0, 1, \dots, w-1\}$ and $h_0^{(w)}(a_0) \neq
0$. If $w>1$ then the $a_i$'s cannot be solved. In this case,
instead let $f = \sum a_i q^{i/w}$. Then replace $q^{1/w}$ with $q$
and re-index $h_j$ as $h_{jw}$ and $h_j = 0$ whenever $j \not\equiv
0$ mod $w$. So we have:
\begin{eqnarray*}
  f(q) &=& \sum a_i q^i\,,\\
  g(x) &=& \sum h_{jw}(x) q^{jw}\,.\\
\end{eqnarray*}
Then $Q_i = 0$ for all $i$ from $1$ to $w-1$, because each of their
terms contains either $h_0^{(j)}(a_0)$ for $j \in [0,w-1]$ or $h_j$
for $j \in [1,w-1]$. The next non-zero term is:
\begin{equation*}
Q_w = h_w(a_0) + a_1^w h_0^{(w)}(a_0)\,.
\end{equation*}

So we solve
\begin{equation*}
 a_1^w = \frac{-h_w(a_0)}{h_0^{(w)}(a_0)}\,.
\end{equation*}

There are two cases:

\emph{Case 1:} If $a_1 \neq 0$ then there are exactly $w$ choices
for $a_1$ and they all differ by an $w$th root of unity. Moreover,
all subsequent $a_i$'s are uniquely determined because for example:
\begin{equation*}
  Q_{w+1} = a_1 h'_w(a_0) + a_1^{w+1}h_0^{(w+1)}(a_0) + a_1^{w-1}a_2
  h_0^{(w)}(a_0)
\end{equation*}
can be solved for $a_2$. And more generally:
\begin{equation*}
  Q_{w+c} = a_1^{w-1} a_{c+1} h_0^{(w)}(a_0) + (\text{terms with
  all $a_i$'s having $i \leq c$}).
\end{equation*}
So we can solve for each $a_{c+1}$.

(In general, when calculating $Q_M$, the partitions that give
(possibly) non-zero terms are partitions $d_1 + \dots + d_n = d$
such that (1) $d \leq M$, (2) $M \equiv d \mod N$, and (3) $n\geq N$
if $M=d$.)

\emph{Case 2:} On the other hand, if $a_1 = 0$, let $\overline{f}(q)
= f(q) + q$ and $\overline{g}(x,q) = g(x-q,q)$. The coefficients of
$\overline{f}$ and $f$ are the same except the $q$-term is non-zero,
and $\overline{g}(\overline{f}(q),q)=0$. If we repeat the above
process on $\overline{g}$ and $\overline{f}$ we go into Case 1 and
get a sequence for $\overline{f}(q)$, and hence $f(q)$. There is
some reindexing involved in this, but note that the ``cusp width''
$w$ remains the same after the reindexing, because replacing $x$
with $x-q$ in $g(x,q) = \sum h_j(x) q^j$ does not change the
$h_0(x)$ term. That is to say, if
\begin{equation*}
  \overline{g}(x,q) = \sum_{j=0}^\infty \overline{h}_j(x) q^j
\end{equation*}
then $\overline{h}_0(x) = h_0(x)$.

Note that in this recursive solving process, we stay in the field
$K'$. That is to say, $a_i \in K'$ for all $i$. This proves
Proposition \ref{cor_to_ctok}.

\section*{Acknowledgements}
The authors would like to thank Prof. Wen-Ching Winnie Li and Prof.
Yifan Yang for their valuable inputs.


\begin{thebibliography}{DLM00}

\bibitem[ASD71]{a-sd}
A.~O.~L. Atkin and H.~P.~F. Swinnerton-Dyer, \emph{Modular forms on
  noncongruence subgroups}, Combinatorics (Proc. Sympos. Pure Math., Vol. XIX,
  Univ. California, Los Angeles, Calif., 1968), Amer. Math. Soc., Providence,
  R.I., 1971, pp.~1--25.

\bibitem[BKO04]{Bruinier-Kohnen-Ono04}
J.~H. Bruinier, W.~Kohnen, and K.~Ono, \emph{The arithmetic of the
values of
  modular functions and the divisors of modular forms}, Compos. Math.
  \textbf{140} (2004), no.~3, 552--566.

\bibitem[Coj05]{Cojocaru05}
A.~C. Cojocaru, \emph{On the surjectivity of the {G}alois
representations
  associated to non-{CM} elliptic curves}, Canad. Math. Bull. \textbf{48}
  (2005), no.~1, 16--31, With an appendix by Ernst Kani.

\bibitem[CP03]{cummins-pauli03}
C.~J. Cummins and S.~Pauli, \emph{Congruence subgroups of {${\rm
PSL}(2,{\Bbb
  Z})$} of genus less than or equal to 24}, Experiment. Math.
  http://www.math.tu-berlin.de/~pauli/congruence/ \textbf{12} (2003), no.~2,
  243--255.

\bibitem[DLM00]{DLM00}
C.~Dong, H.~Li, and G.~Mason, \emph{Modular-invariance of trace
functions in
  orbifold theory and generalized {M}oonshine}, Comm. Math. Phys. \textbf{214}
  (2000), no.~1, 1--56.

\bibitem[Hsu96]{hsu96}
T.~Hsu, \emph{Identifying congruence subgroups of the modular
group}, Proc.
  Amer. Math. Soc. \textbf{124} (1996), no.~5, 1351--1359.

\bibitem[KL81]{Kubert-Lang-81}
D.~S. Kubert and S.~Lang, \emph{Modular units}, Grundlehren der
Mathematischen
  Wissenschaften [Fundamental Principles of Mathematical Science], vol. 244,
  Springer-Verlag, New York, 1981.

\bibitem[KL08]{Kurth-Long06}
C.~A. Kurth and L.~Long, \emph{On modular forms for some
noncongruence
  arithmetic subgroups}, J. of Number Theory \textbf{128} (2008), no.~7,
  1989--2009.

\bibitem[KM03]{knopp-mason99-JNT}
M.~Knopp and G.~Mason, \emph{Generalized modular forms}, J. Number
Theory
  \textbf{99} (2003), no.~1, 1--28.

\bibitem[LLT95]{Lang-cong-algorithm}
M.~L. Lang, C.~H. Lim, and S.~P. Tan, \emph{An algorithm for
determining if a
  subgroup of the modular group is congruence}, J. London Math. Soc. (2)
  \textbf{51} (1995), no.~3, 491--502.

\bibitem[Man72]{Manin72}
J.~I. Manin, \emph{Parabolic points and zeta functions of modular
curves},
  Izv. Akad. Nauk SSSR Ser. Mat. \textbf{36} (1972), 19--66.

\bibitem[MK08]{kohnen-mason08}
G.~Mmason and W.~Kohnen, \emph{On generalized modular forms and
their
  applications}, Nagoya J. (to appear) (2008).

\bibitem[Ser76]{ser76}
J.~P. Serre, \emph{Divisibilit\'e de certaines fonctions
arithm\'etiques},
  Enseignement Math. (2) \textbf{22} (1976), no.~3-4, 227--260.

\bibitem[Shi71]{shim1}
G.~Shimura, \emph{Introduction to the arithmetic theory of
automorphic
  functions}, Publications of the Mathematical Society of Japan, No. 11.
  Iwanami Shoten, Publishers, Tokyo, 1971, Kan\^o Memorial Lectures, No. 1.

\bibitem[Sil86]{sil1}
J.~H. Silverman, \emph{The arithmetic of elliptic curves},
Springer-Verlag, New
  York, 1986.

\bibitem[Tak97]{Takagi97}
T.~Takagi, \emph{The cuspidal class number formula for the modular
curves
  {$X\sb 0(M)$} with {$M$} square-free}, J. Algebra \textbf{193} (1997), no.~1,
  180--213.

\end{thebibliography}

\providecommand{\bysame}{\leavevmode\hbox
to3em{\hrulefill}\thinspace}
\providecommand{\MR}{\relax\ifhmode\unskip\space\fi MR }
\providecommand{\MRhref}[2]{%
  \href{http://www.ams.org/mathscinet-getitem?mr=#1}{#2}
} \providecommand{\href}[2]{#2}

\end{document}